\documentclass{amsart}
\usepackage{amssymb, latexsym, slashbox}
\usepackage{hyperref}
\usepackage{color}

\newcommand{\XREF}{\label}

\newcommand{\R}{{\mathbb R}}
\newcommand{\C}{{\mathbb C}}

\newcommand{\Z}{{\mathbb Z}}

\newcommand{\CC}{{\widehat \C}}

\newtheorem{theorem}{Theorem}
\newtheorem{proposition}{Proposition}

\theoremstyle{definition}
\newtheorem{definition}{Definition}
\theoremstyle{remark}
\newtheorem{remark}{Remark}

\newtheorem{example}{Example}
\newtheorem{question}{Open Question}

\numberwithin{equation}{section} \numberwithin{theorem}{section}
\numberwithin{definition}{section} \numberwithin{remark}{section}
\numberwithin{example}{section} \numberwithin{lemma}{section}
\numberwithin{property}{section}
\numberwithin{proposition}{section} \numberwithin{claim}{section}
\numberwithin{othertheorem}{section} \numberwithin{conj}{section}
\numberwithin{corollary}{section}

\begin{document}

\title{Hereditarily non uniformly perfect sets}

\author{Rich Stankewitz}
\address[Rich Stankewitz]{Department of Mathematical Sciences\\
    Ball State University\\
    Muncie, IN 47306}
\email{rstankewitz@bsu.edu}

\author{Toshiyuki Sugawa}
\address[Toshiyuki Sugawa]{
Graduate School of Information Sciences\\
Tohoku University\\
Sendai 980-8578, Japan}
\email{sugawa@math.is.tohoku.ac.jp}

\author{Hiroki Sumi}
\address[Hiroki Sumi]
    {Department of Mathematics\\
    Graduate School of Science\\
    Osaka University\\
    1-1, Machikaneyama, Toyonaka\\
    Osaka, 560-0043, Japan\\
    http://www.math.sci.osaka-u.ac.jp/$\sim $sumi/welcomeou-e.html}
\email{sumi@math.sci.osaka-u.ac.jp}

\begin{abstract}
We introduce the concept of hereditarily non uniformly perfect sets, compact sets for which no compact subset is uniformly perfect, and compare them with the following: Hausdorff dimension zero sets, logarithmic capacity zero sets, Lebesgue 2-dimensional measure zero sets, and porous sets.  In particular, we give a detailed construction of a compact set in the plane of Hausdorff dimension 2 (and positive logarithmic capacity) which is hereditarily non uniformly perfect.
\end{abstract}

\maketitle

\section{Introduction and results}

Various types of non-smooth sets arise naturally in many mathematical settings.  Julia sets, attractor sets generated from iterated function systems, bifurcation sets in parameter spaces such as the boundary of the Mandelbrot set, and Kakeya sets are all examples of non-smooth sets which are studied intensely.  Classical analysis has given way to fractal geometry for the purposes of studying such ``pathological" sets.  Many tools such as Hausdorff measure and Hausdorff dimension, logarithmic capacity, porosity, Lebesgue measure, and uniform perfectness have been utilized to discern certain fundamental ``thickness/thinness" properties of such pathological sets.  In this paper we compare how the above tools and properties relate to each other with regard to compact sets in the complex plane $\C$.  In particular, we are interested in how the thinness properties with respect to each of these notions relate.  We will consider compact sets $E \subset \C$ and study the following conditions: $\dim_H E=0$,  Cap $E = 0$, $m_2(E)=0$, $E$ is porous.  However, these four properties are hereditary properties of thinness in the sense that if $E$ satisfies one of these properties, then all subsets of $E$ also satisfy the same property.  Uniform perfectness of a (compact) set $E$ (see definition below) is a property quantifying a \textit{uniform} thickness near each point of $E$.  To get at a compatible notion of thinness we need to require more than just that the set fails to be uniformly perfect.  For example, a set $E=F \cup \{z_0\}$, where $F$ is uniformly perfect and $z_0 \notin F$, fails to be uniformly perfect (since $z_0$ is isolated), yet $E$ is ``thick" near all of the points of $F$.  Thus to capture the correct idea of being ``thin", as a counterpart to the uniformly perfect notion of ``thick", we offer the following.

\begin{definition}
A compact set $E$ is called \textit{hereditarily non uniformly perfect} (HNUP) if no subset of $E$ is uniformly perfect.
\end{definition}

The main results of this paper are stated below as Theorems~\ref{TableThm},~\ref{dim=1},  and~\ref{dim=2}.

For Theorem~\ref{TableThm}, the relevant definitions are found in Section~\ref{SecDefs} and the superscripts correspond to the proofs given in Section~\ref{proofs}.

\begin{theorem}\label{TableThm}
The following implications indicated in Table~\ref{Table}, showing when $X$ implies $Y$ in the case that $E \subset \C$ is a compact set, all hold.
\begin{table}[here]
\begin{tabular}{|c|c|c|c|c|c|c|c|}
  \hline
   \backslashbox[0mm]{$Y$}{$X$} &  $\dim_H E=0$ &  Cap $E = 0$  & $E$ is HNUP &  $m_2(E)=0$ & $E$ is porous \\
  \hline

   $\dim_H E=0$ & $\ast$ &  $yes^1$ &  $no^2$ &  $no^3$ &  $no^4$   \\
  Cap $E = 0$& $no^5$ & $\ast$ & $no^6$ & $no^7$ & $no^8$ \\
  $E$ is HNUP & $yes^9$ & $yes^{10}$ & $\ast$ & $no^{11}$ & $no^{12}$  \\
  $m_2(E)=0$ & $yes^{13}$ & $yes^{14}$ & $no^{15}$ & $\ast$ & $yes^{16}$  \\
  $E$ is porous  & $no^{17}$ & $no^{18}$ & $no^{19}$ & $no^{20}$ & $\ast$  \\

  \hline
\end{tabular}
  \caption{Does $X$ imply $Y$ when $E \subset \C$ is a compact set?
}\label{Table}
\end{table}
\end{theorem}

In Section~\ref{ex} we use a straightforward Cantor-like construction to build sets with certain properties that are then assembled in such a way to justify Theorems~\ref{dim=1} and~\ref{dim=2}, showing all necessary details.  Instead of using a constant ratio of lengths of basic intervals in subsequent stages, we carefully choose sequences of these ratios that vary in a way which provide the desired properties.  Those familiar with non-autonomous IFS theory will recognize a connection as our sets are often attractor sets of such non-autonomous IFS's.

\begin{theorem}\label{dim=1}
There exists a compact HNUP set $W \subset \R$ such that $\dim_H W=1, H^1(W)=0$ and Cap $W>0$.  Hence, for any $0< s <1$ and for any $0<c<\infty$, there exists a compact HNUP subset $K$ of $W$ such that $\dim _{H}(K)=s$ and $H^{s}(K)=c$.

Moreover, we can arrange the set $W$ so that $W$ satisfies the following additional property. For each $x\in W$ and for each $M>1$, there exist two positive numbers $r, R$ with $R/r\geq M$ such that $\{ y\in \Bbb{R} \mid r<|y-x|<R\} \subset \Bbb{R}\setminus W$ and $\{ y\in \Bbb{R} \mid |y-x|\geq R\} \cap W\neq \emptyset .$
\end{theorem}

\begin{theorem}\label{dim=2}
There exists a compact HNUP set $E \subset \C$ such that $\dim_H E=2, m_2(E)=0$ and Cap $E>0$.  Hence, for any $0< s<2$ and for any $0<c<\infty$, there exists a compact HNUP subset $K$ of $E$ such that $\dim _{H}(K)=s$ and $H^{s}(K)=c.$

Moreover, we can arrange the set $E$ so that $E$ satisfies the following additional property. For each $x\in E$ and for each $M>1$, there exist two positive numbers $r, R$ with $R/r\geq M$ such that $\{ y\in \Bbb{C} \mid r<|y-x|<R\} \subset \Bbb{C}\setminus E$ and $\{ y\in \Bbb{C} \mid |y-x|\geq R\} \cap E\neq \emptyset .$
\end{theorem}

\begin{remark}
Theorem 5.4 in~\cite{FalcGeomFractSets} says that if closed set $E \subset \R^n$ has $H^s(E) = +\infty$ for some $s>0$, then for each $c \in (0,+\infty)$ there exists a compact $K \subset E$ with $H^s(K)=c$ and thus $\dim_H K =s$.  This provides the justification for the second sentences in Theorems~\ref{dim=1} and~\ref{dim=2}.
\end{remark}

With the use of far deeper results requiring much more machinery than for our proofs of Theorems~\ref{dim=1} and~\ref{dim=2}, we note that Examples~\ref{m1>0} and~\ref{m2>0} shows that there exist HNUP sets $\widetilde{W} \subset \R$ and $\widetilde{E} \subset \R^2$ with $H^1(\widetilde{W})>0$ (and thus $\dim_H \widetilde{W}=1$ and Cap $\widetilde{W}>0$) and $m_2(\widetilde{E})>0$ (and thus $\dim_H \widetilde{E}=2$ and Cap $\widetilde{E}>0$).  Note that the set $\widetilde{W}$ in Example~\ref{m1>0} does
not satisfy the additional property in Theorem~\ref{dim=1} (since, by the Lebesgue Density Theorem, any set satisfying that additional property must have Lebesgue measure zero). Similarly, $\widetilde{E}$ does not satisfy the additional property in Theorem~\ref{dim=2}.  Lastly, we note that it remains an interesting open problem to produce more elementary constructions of such sets as in Examples~\ref{m1>0} and~\ref{m2>0}.

\section{Definitions and basic facts}\label{SecDefs}

This section contains brief introductions to the definitions and basic facts relating to Hausdorff measure and Hausdorff dimension, logarithmic capacity, porosity, and uniform perfectness.  We assume that the reader is familiar with the basic properties of $m_2$, Lebesgue 2-dimensional measure.  Our primary reference for Hausdorff measure and dimension is~\cite{Falc}, which the reader may wish to have available to see all the details (including for formulas used in the constructions in Section\ref{ex}).

\subsection{Hausdorff measure and Hausdorff dimension}
Let $h:(0,+\infty) \to (0,+\infty)$ be a dimension function, i.e., a non-decreasing continuous function such that $\lim_{t\to 0^+} h(t)=0$.  The \textit{Hausdorff h-content} $\mathcal{L}_h(E)$ of a set $E\subset \C$ is defined as the infimum of the sum $\sum_k h(d(B_k))$ taken over all countable covers of $E$ by sets $B_k$.  Here and throughout $d(A)$ denotes the diameter of the set $A$ in the Euclidean metric.  The \textit{Hausdorff h-measure} of a set $E$ is given by $H_h(E)=\lim_{\epsilon \to 0^+} \inf\{\sum_k h(d(B_k))\}$ where the infimum is taken over countable covers of $E$ by sets $B_k$ such that each $d(B_k) < \epsilon$.  For our purposes we will only use dimension functions of the form $h(t)=t^\alpha$ (with $\alpha > 0$) and we shall employ the notation $\mathcal{L}^\alpha(E)=\mathcal{L}_h(E)$ and $\mathcal{H}^\alpha(E)=\mathcal{H}_h(E)$.  (Also, $\mathcal{H}^0$ corresponds to the counting measure.)  We note that with some effort one can show that $\mathcal{L}_h(E)=0$ if and only if $\mathcal{H}_h(E)=0.$  A simple calculation (see~\cite{Falc}, Ch.~2) shows that if $\mathcal{H}^\alpha(E)<\infty$, then $\mathcal{H}^\beta(E)=0$ for all $\beta>\alpha$.  For a subset $E \subset \C$ the Hausdorff dimension of $E$ is given as $\dim_H E=\sup\{\alpha:\mathcal{H}^\alpha(E)=+\infty\}=\inf\{\alpha:\mathcal{H}^\alpha(E)=0\}$.

Also found in~\cite{Falc} are the facts that (i) $\dim_H C =0$ when $C$ is countable, (ii) $\dim_H F \leq \dim_H E$ whenever $F \subseteq E$ (monotonicity), (iii) $\dim_H E = \sup_m \dim_H E_m$ when $E = \cup_{m=1}^\infty E_m$, and (iv) $\dim_H E = dim_H E'$ when $E'$ is the image of $E$ under a similarity transformation.

\vskip.2in

\subsection{Uniform perfectness}
The notion of uniform perfectness was first introduced by Beardon and Pommerenke in~\cite{BP}.  A set is called perfect if it has no isolated points, whereas the notion of uniform perfectness is a quantified version of perfectness.  Uniformly perfect sets can be equivalently defined in many different ways and be defined for sets in $\R^n$, but we are mainly concerned with sets in the complex plane $\C$ and so we define uniformly perfect sets as follows.

\begin{definition} \XREF{sepann}
For $w \in \C$ and $r, R >0$ the true annulus $A=Ann(w;r,R)=\{z:r < |z-w| < R\}$ is said to
separate a set $F \subset \C$ if $F$ intersects both components of $\C \setminus
A$ and $F \cap A = \emptyset$.
\end{definition}

\begin{definition} \XREF{mod}
The modulus of a true annulus $A=Ann(w;r,R)$ is $\mod A=\log(R/r)$.
\end{definition}

\begin{definition} \XREF{updef}
A compact subset $F \subset \C$ with two or more points is uniformly perfect\footnote{More generally, a compact subset $F \subset \R^n$ is uniformly perfect if there exists a constant $c>0$ such that $F\cap Ann(a; cr,r)\ne \emptyset$ for any $a \in F$ and $0<r<d(F)$.  Here, of course, the Euclidean metric in $\R^n$ is used to define the annular region $Ann(a; cr,r)$.} if
there exists a uniform upper bound on the moduli of all annuli which separate $F$.
\end{definition}


Thus a set is uniformly perfect if there is a uniform bound on how ``fat" (large modulus) an annulus can be and still separate the set.  Uniform perfectness, in a sense, measures how ``thick" a set is near each of its points and is related in spirit to many other notions of thickness such as Hausdorff content and dimension, logarithmic capacity and density, H\"older regularity, and positive injectivity radius for Riemann surfaces.  For an excellent survey of uniform perfectness and how it relates to these and other such notions see Pommerenke~\cite{P} and Sugawa~\cite{Sug}.  In particular, we note that uniformly perfect sets are necessarily uncountable.

%
%

\subsection{Logarithmic capacity}
The logarithmic capacity of a compact set $E \subset \C$ can also be defined in many different, yet equivalent, ways (see~\cite{Fi}, Ch.~1).  For example, one may define it in terms of the asymptotic behavior of Green's function defined on $\CC \setminus E$ with pole at $\infty$ or in terms of the infimum of an energy integral.  We shall employ the following definition given via the transfinite diameter.  Set $P_n=\max \prod_{j<k} |z_j-z_k|$ where $z_j \in E$ for $j=1, \dots, n$.  The sequence $D_n=P_n^{2/n(n-1)}$ is non-increasing (see~\cite{A}, p.~23) and the logarithmic capacity is defined by Cap $E=\lim_{n \to \infty} D_n$.

We note that capacity is monotone, i.e., Cap $ F \leq \textrm{Cap } E$ whenever $F \subseteq E$, and Cap $ C =0$ when $C$ is countable.  The countability result follow from the obvious fact that a finite set has zero capacity and that for Borel sets $B_1 \subseteq B_2 \subseteq \dots$ with $B=\cup B_n$, we have Cap $B= \lim_{n \to \infty}$ Cap $B_n$ (see Theorem 5.1.3 in~\cite{Ransford}).

\subsection{Porosity}
For $a \in \C$ and $r>0$, let $B(a,r)=\{z:|z-a|<r\}$.  A compact set $E \subset \C$ is defined to be \textit{porous} if there exists a constant $c>0$ such that for any point $a \in E$ and radius $r>0$ there exists a ball $B(b,cr) \subset B(a,r)$ such that $B(b,cr) \cap E = \emptyset$.

\section{Proofs for results in Table~\ref{Table}}\label{proofs}

Here we present the proofs for the statements made in Table~\ref{Table}.  Many of these results are known (being based on classical theorems), but since the arguments are short we include them here for completeness. We also note that most results apply to $\R^n$ as well, but since that is not the focus of this paper, we do not detail arguments for such.
\vskip.2in
(1) This follows from Frostman's Theorem (see~\cite{Ts}, p.~65) which states: If $h(t)$ is a dimension function such that $\int_0^1 h(t)/t \, dt < +\infty$ and $\mathcal{H}_h(E)>0$, then Cap $E>0$.  Note that $h(t)=t^{\alpha}$, where $\alpha >0$, satisfies the integral condition.  Thus if Cap $E=0$, then $\mathcal{H}^{\alpha}(E)=0$ for all $\alpha>0$ and so $\dim_H E=0$.

\begin{remark}
The integral condition $\int_0^1 h(t)/t \, dt < +\infty$ in the Frostman theorem is crucial.  for example, if $h(t)=-1/\log t$ and $0<\mathcal{H}_h(E)<\infty$, then Cap $E=0$ (see~\cite{Ts}, p.~66).
\end{remark}

(2) Example~\ref{hdim} below gives a compact set $I$ which is HNUP yet $\dim_H I >0$.

(3) Since $H^2(E) = (4/\pi) m_2(E)$ (see~\cite{Falc}, Ch.~3), any compact set $E \subset \C$ with $0<\dim_H(E)<2$ verifies the assertion.  The middle third Cantor set is a specific example which is known to have Hausdorff dimension $\log 2/\log 3$.

(4) Any compact set $E \subset \R$ with $\dim_H(E)>0$ verifies this claim since for $x \in \R$ each ball $B(x,r) \supset B(x+ir/2,r/2)$, yet $B(x+ir/2,r/2)\cap E = \emptyset.$

(5) Example~\ref{hdim=0} below gives a compact set $I$ such that $\dim_H I=0$ yet Cap $I >0$.

(6) Example~\ref{hdim} below gives a compact set $I$ which is HNUP yet Cap $I >0$.

(7) Any compact interval of length $L$ is known to have capacity $L/4$ (see~\cite{Ts}, p.~84), yet has zero two dimensional measure.

(8) Any compact set $E \subset \R$ with Cap $E>0$ verifies this claim since each ball $B(x,r) \supset B(x+ir/2,r/2)$, yet $B(x+ir/2,r/2)\cap E = \emptyset.$

(9) Since Hausdorff dimension is monotone, the claim follows from the known fact that uniformly perfect sets necessarily have positive Hausdorff dimension.  In fact, J\"arvi and Vuorinen (see~\cite{JV}, p.~522) have shown that a compact set $E \subset \C$ is uniformly perfect if and only if there exist constants $C>0$ and $\alpha >0$ such that $\mathcal{L}^\alpha(E \cap B(a,r)) \ge Cr^\alpha$ for all $a \in E$ and $0<r<d(E)/2$.

(10) Since uniformly perfect sets necessarily have positive capacity, the claim follows from the monotonicity of capacity.  In particular, Pommerenke~\cite{P} has shown that a compact set $E \subset \C$ is uniformly perfect if and only if there exists $c>0$ such that Cap $(B(z,r) \cap E) \ge cr$ for all $z \in E$ and $0<r<d(E)$.

(11) Any compact interval has zero two dimensional measure and is trivially uniformly perfect since it is connected (and therefore cannot be separated by any annulus).

(12) Any compact interval of the real line is trivially uniformly perfect since it is connected, yet is also porous with constant $c=1/2$ as in (4).

(13) From the definition, $\dim_H E <2$ implies $m_2(E) = (\pi/4) H^2(E)=0$.

(14) It is known that Cap $E \ge \sqrt{m_2(E)/\pi e}$ (see~\cite{Ts}, p.~58) and so the assertion follows.


(16) Given a point $a$ in a porous set $E$ and radius $r>0$, we must have $m_2(B(a,r) \cap E) \le m_2(B(a,r))-m_2(B(b,cr))=\pi r^2-\pi c^2r^2=\pi r^2(1-c^2)$ where $c>0$ is the constant in the definition of porosity.  The Lebesgue Density Theorem implies that $\frac{m_2(B(a,r) \cap E)}{\pi r^2}\to 1$ as $r \to 0$ for a.a. $a \in E$.  Thus we must have $m_2(E)=0$.  In fact, a stronger result holds (see~\cite{Falc}, Proposition 3.12) which shows that a porous set $E \subset \C$ must have $\dim_H E <2$.

(20) For $a \in \C$ and $r>0$, let $C(a,r)=\{z:|z-a|=r\}$.  Consider the set $E=\overline{\cup_{n=1}^\infty C(0,1/n)}$.  The ball $B(\frac{n+\frac12}{n(n+1)},\frac{1}{2n(n+1)})$ is a ball of largest radius which is contained in both $\C \setminus E$ and $B(0,1/n)$.  Since the ratio of the radii $\frac{{1}/{2n(n+1)}}{{1}/{n}}=\frac{1}{2(n+1)} \to 0$ as $n \to \infty$, we see that $E$ fails to be porous at the origin.  Thus the set $E$ is not porous, yet clearly $m_2(E)=0$.

(17-19) Replace each circle in (20) by a ``discrete circle" of 100,000 points equally spaced on the given circle to obtain a set that is not porous (at the origin).  Since this set is countable it must have Hausdorff dimension zero, logarithmic capacity zero, and be HNUP.

(15) We now present Examples~\ref{m1>0} and~\ref{m2>0}, and thus justify item (15) in Table~\ref{Table}.

But first we must offer here our special thanks to the first referee for guiding us towards the results found in~\cite{McMullen} and~\cite{BKFRW}, indicating how they could be used to answer this question which had been left unsolved in our earlier draft.  Credit for this result is therefore due to this referee whom we thank for allowing us to include it here.  The results used in the construction are not elementary and so instead of providing all details as with most of our other results in this paper, we refer to the salient results of the aforementioned papers.
We also point the interested reader to~\cite{FSU} for more information on Diophantine approximation in the (far more general) setting of hyperbolic metric spaces; its introduction alone contains
a wealth of material to orient the reader in this area.

We use the notation and terminology from~\cite{McMullen}, where, for any lattice $\Gamma \in \textrm{Isom}(\mathbb{H}^{n+1})$ with cusp at infinity, there is a corresponding \emph{Diophantine set} $D(\Gamma) \subset \R^n$ consisting of the endpoints of lifts of all bounded
geodesic rays in $\mathbb{H}^{n+1}/\Gamma$.
In general, $D(\Gamma)$ consists of the points in $\R^n$ which are \emph{badly approximable} by the cusps of $\Gamma$.  For example, as noted in~\cite{McMullen}, $D(SL_2(\Z))=\{x \in \R:\textrm{ for some }c>0, | x-\frac{{p}}{q}|  \geq \frac{c}{q^{2}}\textrm{ for all }p/q \in \mathbb{Q} \}$ is the set of real numbers which are badly approximable by rationals.  In general then we refer to $W(\Gamma)=\R^n \setminus D(\Gamma)$ as the set of points in $\R^n$ which are \emph{well approximable} by the cusps of $\Gamma$.

McMullen proves (see Theorem 1.3 in~\cite{McMullen}) that any such $D(\Gamma) \subset \R^n$ (for any lattice $\Gamma \subset \textrm{Isom}(\mathbb{H}^{n+1})$ with cusp at infinity) is \emph{absolutely winning}, which means it is 0-\emph{dimensionally absolute winning} in the terminology of p.~323 of~\cite{BKFRW}.  Hence, any uniformly perfect subset of $\R^n$, being a 0-\emph{dimensionally diffuse set} in the terminology of Definition 4.2 of~\cite{BKFRW}, must meet $D(\Gamma)$ (see Theorem 4.6 and Proposition 4.9 of~\cite{BKFRW} where even more is shown).  Note that by Theorem 1.5 in~\cite{McMullen} for each lattice $\Gamma$ with cusp at infinity its Diophantine set $D(\Gamma) \subset \R^n$ is $\sigma$-porous and thus has $n$-dimensional Lebesgue measure zero.  By the inner regularity of the Lebesgue measure, $W(\Gamma)$ then contains a compact subset $X$ of positive measure.  Thus any such compact subset $X \subset W(\Gamma)\subset \R^n$ will provide an example of the type we seek, i.e., have positive measure and be HNUP.

\begin{example}\label{m1>0}
Any compact subset $\widetilde{W}$ of \emph{well approximable} numbers $W(SL_2(\Z))$ of positive $H^1$ measure must be HNUP as any uniformly perfect set must meet the complement of $\widetilde{W}$ in the set $D(SL_2(\Z))$.  Clearly, $\dim_H \widetilde{W}=1$, which implies Cap$ \widetilde{W} >0$ as in (1) in Table~\ref{Table}.
\end{example}

\begin{example}\label{m2>0}
When $\Gamma$ is a nonuniform Kleinian lattice of M\"obius maps, its \emph{Diophantine set} $D(\Gamma) \subset \R^2$ is absolutely winning. Since such a set $D(\Gamma)$ has zero $m_2$ measure, any positive $m_2$-measure compact subset $\widetilde{E}$ of \emph{well approximable} points $W(\Gamma) \subset \R^2$ must be HNUP.  Clearly, as above we have $\dim_H \widetilde{E}=2$ and Cap $\widetilde{E}>0$.
\end{example}

\section{Proofs of Theorems~\ref{dim=1} and~\ref{dim=2}}\label{ex}

We begin by considering the following Cantor-like construction in the real line.  Fix $m \in \{2, 3, \dots\}$ and choose $0< a \leq \frac{1}{m+1}$.  Fix a sequence $\bar{a}=(a_1, a_2, \dots)$ such that $a_k \leq a$ for $k=1, 2, \dots$.  We define sets $[0,1]=I_0 \supset I_1 \supset I_2 \supset \dots$ such that each $I_k$ is a union of $m^k$ disjoint closed intervals of the same length (called basic intervals).  Each basic interval $J$ in $I_k$ will contain exactly $m$ equally spaced basic intervals in $I_{k+1}$ such that $I_{k+1}$ contains the endpoints of $J$.  We use the sequence $a_k$ to determine the scaling factor for the lengths of the basic intervals in $I_k$ as compared to the lengths of the basic intervals in $I_{k+1}$.  More precisely, setting $A_0=1$ each of the $m^k$ basic intervals in $I_k$ has length $A_k=a_1 \dots a_k$.  Call $B= 1-ma$, noting $a \leq B<1$ and that since each $a_k \leq a$ we have $1-m a_k \geq B$ for all $k$.   Thus the $m-1$ ``gaps'' between basic intervals in $I_k$ which are contained in the same basic interval from $I_{k-1}$ each have length $e_k$ where
\begin{equation}\label{gap}
(m-1)e_k=A_{k-1}-mA_k =A_{k-1}-ma_k A_{k-1} = A_{k-1}(1-ma_k)\ge  B A_{k-1}.
\end{equation}
Note also that $e_{k+1}=\frac{1}{m-1} A_k(1-ma_{k+1})=\frac{1}{m-1} a_k A_{k-1}(1-ma_{k+1}) < \frac{a}{m-1}  A_{k-1}\leq \frac{B}{m-1}  A_{k-1} \le e_k$ and thus $e_k$ strictly decreases to $0$.  Note that, because the gaps $e_k$ are decreasing, the distance between any basic subintervals of $I_k$ must be separated by a distance at least $e_k$, whether or not these basic subintervals come from the same (``parent") basic interval from $I_{k-1}$.

Define the set
\begin{equation}\label{Idef}
I=I_{\bar{a}}=\bigcap_{k=1}^\infty I_k.
\end{equation}

Since by \eqref{gap} the basic intervals in $I_k$ are ``equally spaced",
 Example~4.6 (and the discussion on p.~59) of \cite{Falc} yields that
\begin{equation}\label{dimformula}
\dim_H I= \liminf_{k \to \infty}
\frac{\log m^{k}}{-\log A_k}.
\end{equation}

\begin{remark}
Formula~\eqref{dimformula} is based on the existence of a mass distribution (measure) as constructed in the discussion prior to Proposition 1.7 of ~\cite{Falc}.  The construction, however, is known not to work in the full generality as stated in~\cite{Falc}, but it does work in the situation we present here since we are in the nice situation that each basic interval is compact.  See~\cite{Falkner}
by N.~Falkner, a Mathematical Review that describes how the general construction of the measure can fail and also states some sufficient conditions to ensure success.
\end{remark}

Key to our results is part (1) of the following dichotomy.
\begin{theorem}\label{hnup}
  For $m \in \{2, 3, \dots\}$ and $\bar{a}=(a_1, a_2, \dots) \in (0, \frac{1}{m+1}]^\mathbb{N}$, let $I_{\bar{a}}$ be given as in~\eqref{Idef}.  Then
  \begin{enumerate}
  \item if $\liminf a_k = 0$, then $I_{\bar{a}}$ is HNUP, and
  \item if $\liminf a_k > 0$, then $I_{\bar{a}}$ is uniformly perfect.
  \end{enumerate}
\end{theorem}

\begin{proof}
Call $I=I_{\bar{a}}$.

To prove (1) suppose $a_{k_n} \to 0$.  Note that each basic interval $J$ of $I_k$ is contained in the bounded component of
the complement of the annulus $C=Ann(x_J;r_k,R_k)$ where $r_k=A_k/2$,
$R_k=r_k +e_k$ and $x_J$ denotes the midpoint of $J$.  Since the gaps $e_k$ are decreasing,
$C$ separates $I$.  Along the sequence $k_n$ we use~\eqref{gap} to see
$\frac{R_{k_n}}{r_{k_n}} =
1+ \frac{e_{k_n}}{r_{k_n}} =
1+ 2\frac{e_{k_n}}{A_{k_n}} =
1+ 2\frac{A_{{k_n}-1}-m A_{k_n}}{(m-1)A_{k_n}}=
1+ \frac{2}{(m-1)}(\frac{1}{a_{k_n}}-m)
\to +\infty$ as $n \to \infty$.
Thus for each point $x \in I$ there exists an annulus $C$ of arbitrarily large modulus with arbitrarily small outer radius $R_k$ which separates $I$ and has $x$ in the bounded component of the complement of $C$. Hence, $x$ can never be a point in a uniformly perfect subset of $I$ and thus no subset of $I$ can be uniformly perfect.

To prove (2) suppose $\liminf a_k > 0$.  Thus $\delta = \inf a_k>0$ and we set $M=1+ \frac{2}{m-1}\left(\frac1{\delta} - m\right)$.  Suppose $C=Ann(z;r,R)$ is an annulus which separates $I$.  We show that $R/r \leq M$.  Since $C$ separates $I$, we must have that $C$ separates $I_k$ for large $k$ (in particular, whenever $A_k < R-r$ any basic interval of $I_k$ meeting $C$ would have one of its endpoints in $C$ thus showing that $C$ does not separate $I$ since such an endpoint is also in $I$).  Let $n$ be the smallest positive integer such that $C$ separates $I_n$.  Considering $C \cap \R$ and the fact that $C$ separates $I_n$ but does not separate $I_{n-1}$, we see that $R-r \leq e_n$.  Since $B(z,r)$ contains some basic interval of $I_n$ (of length $A_n$), we have $2r \geq A_n$.  Hence using~\eqref{gap} we see $\frac{R}{r} \leq \frac{r+e_n}{r} \leq 1 +\frac{e_n}{A_n/2}=1+ 2\frac{A_{n-1}-m A_n}{(m-1)A_n}=
1+ \frac{2}{(m-1)}(\frac{1}{a_n}-m)
\leq 1+ \frac{2}{m-1}\left(\frac1{\delta} - m\right)=M.$
\end{proof}

\begin{remark}
The artificial-looking condition $0<a \le 1/(m+1)$ in Theorem~\ref{hnup} was chosen to simplify the related proofs (by, in particular, to forcing the gap sizes to be decreasing), but it is worth noting that it can be relaxed to $0<a<1/m$ though we shall omit its more involved proof.
\end{remark}

\begin{example}\label{hdim} Fix $m \in \{2, 3, 4, \dots\}$ and choose $0< a \leq \frac{1}{m+1}$. Define $a_k=a^{n}$ if $k=2^n$ for $n=1, 2, \dots$ and $a_k=a$ otherwise.  Then the set $I$ given in~\eqref{Idef} satisfies the following:

a) $\dim_H I=\frac{\log m}{-\log a}$,

b) Cap $I>0$,

c) $I$ is HNUP.
\end{example}

\begin{remark}
We note that, if we select instead $a_k=a$ for \textit{all} $k$, then by~\eqref{dimformula} $\dim_H I$ would still equal $\frac{\log m}{-\log a}$ and by Theorem~\ref{hnup}(2) $I$ would be uniformly perfect.
Thus the effect of reducing the total length of the basic intervals in step $2^n$
by a very small factor of $a^n$ times those in the previous step $2^n-1$, but only
doing this sparingly at these times $2^n$, is to ensure that these sets are thin in
the sense of being HNUP, but not thin in the sense of Hausdorff dimension
as $\dim_H I> 0.$
\end{remark}

\begin{proof}[Proof of (a) in Example~\ref{hdim}]


Let $k$ be a positive integer and let $n_0$ be the integer such that
$2^{n_0}\leq k < 2^{n_0+1}$.  Thus we have $ \frac{\log k}{\log 2}-1 < n_0 \leq \frac{\log k}{\log 2}$.
Since
$ \log A_k
=\log(a_1\dots a_k)=\sum_{j=1}^k \log a_j
=\sum_{j=1,a_j=a}^k\log a + \sum_{j=1,a_j\neq a}^k\log a_j
=(k-n_0) \log a + \sum_{p=1}^{n_0} \log(a^{p})
=(k-n_0) \log a + \sum_{p=1}^{n_0} p \log a
=(k-n_0) \log a + \frac{n_0(n_0+1)}{2}\log a$,
we have
\begin{equation}\label{logA1}
\left(k-\frac{\log k}{\log 2}+1\right) \log a +
\frac{1}{2}\left(\frac{\log k}{\log 2}\right) \left(\frac{\log k}{\log 2}+1\right) \log a < \log A_k
\end{equation}
and
\begin{equation}\label{logA2}
 \log A_k < \left(k-\frac{\log k}{\log 2}\right) \log a
+\frac{1}{2}\left(\frac{\log k}{\log2}-1\right)\left(\frac{\log k}{\log2}\right)\log a.
\end{equation}
Thus by~\eqref{dimformula} $\dim_H I= \liminf_{k \to \infty}
\frac{k\log m}{-\log A_k} = \frac{\log m}{-\log a}$.
\end{proof}

\begin{proof}[Proof of (b) in Example~\ref{hdim}]
This follows from (1) in Table~\ref{Table}.
\end{proof}

\begin{proof}[Proof of (c) in Example~\ref{hdim}]
This follows from Theorem~\ref{hnup}(1) since $a_{2^n}=a^n \to 0$.
\end{proof}

The following result shows that we have lots and lots of examples where $I_{\bar{a}}$ is HNUP and has positive Hausdorff dimension.

\begin{theorem}
  Let $\lambda_m$ be normalized Lebesgue measure on $(0,\frac{1}{m+1}]$ so that it is a probability measure and call $\widetilde{\lambda_m}:= \bigotimes_{n=1}^\infty \lambda_m$ be the product measure on $X:= \prod_{n=1}^\infty (0, \frac{1}{m+1}]$.  Then for $\widetilde{\lambda_m}$-a.a. $\bar{a}=(a_1, a_2, \dots) \in X$, $I_{\bar{a}}$ is HNUP and $\dim_H I_{\bar{a}} = \frac{\log m}{\log(m+1)+1}>0$.
\end{theorem}

\begin{proof}
  Considering the i.i.d. random variables $a_n$ we note that the strong law of large numbers (or the ergodicity of the shift map on $X$ with respect to
$\widetilde{\lambda_{m}}$) applied to $\log a_n$ shows that for $\widetilde{\lambda_m}$-a.a. $\bar{a}=(a_1, a_2, \dots) \in X$ we have $\frac1k \log(a_1 \dots a_k) = \frac1k \sum_{j=1}^k \log a_j \to \mathbb{E}(\log a_1) = (m+1) \int_0^{\frac{1}{m+1} } \log x \, dx = -(\log (m+1) +1)$.  Hence by~\eqref{dimformula}, $\dim_h I_{\bar{a}}= \liminf_{k \to \infty}
\frac{\log m^{k}}{-\log A_k} = \liminf \frac{k \log m}{-\log(a_1 \dots a_k)} =\frac{\log m}{\log(m+1)+1}$.

Also, we see that by Theorem~\ref{hnup}(1), $I_{\bar{a}}$ is HNUP whenever $\liminf a_k =0$, which as we now demonstrate occurs for $\widetilde{\lambda_m}$-a.a. $\bar{a}=(a_1, a_2, \dots) \in X$.  For each positive integer $n$, let
$D_{n}$ be the set of $\bar{a}$ such that
$a_{j}> \frac{1}{n}$ for all $j$
and note $\widetilde{\lambda_{m} }(D_{n})=0.$
Since $\{ \overline{a}\mid \liminf a_{k}>0\}
= \{ \overline{a}\mid \inf a_{k}>0\}
=\cup _{n=1}^{\infty }D_{n}$, we obtain that
$\widetilde{\lambda_{m} }(\{ \overline{a}\mid \liminf a_{k}>0\})=0.$
\end{proof}

\begin{example}\label{hdim=0} Using $m=2, a=1/4 \leq 1/(m+1)$, and $a_k=a^{k}$ for $k=1, 2, \dots$,  the Cantor-like set $I$ given in~\eqref{Idef} satisfies the following:

a) $\dim_H I=0$,

b) Cap $I>0$,

c) $I$ is HNUP.

\end{example}

\begin{proof}[Proof of (a) in Example~\ref{hdim=0}]
Set $A_0=1$ and note that each of the $2^k$ basic intervals in $I_k$ has length $A_k=a_1 \dots a_k=a^{1}a^{2}a^{3} \dots a^{k}=a^{\sum_{j=1}^k j}=a^{k(k+1)/2}=a^{(k^2+k)/2}$.

By Proposition~4.1 of~\cite{Falc},
$$\dim_H(I)\leq \liminf_{k \to \infty}
\frac{\log2^{k}}{-\log A_k}=\liminf_{k \to \infty}
\frac{\log2^{k}}{-\log a^{(k^2+k)/2}}=\liminf_{k \to \infty}
\frac{2k\log2}{-(k^2+k)\log a}=0.$$
\end{proof}

\begin{proof}[Proof of (b) in Example~\ref{hdim=0}]
Let $E_k$ be the set of $2^{k+1}$ endpoints of the $2^{k}$  basic intervals of $I_{k}$.  For disctinct points $z,z' \in I$, let $\mu(z,z')=\max\{k:z \textrm{ and } z' \textrm{ lie in the same basic interval of } I_{k}\}.$  Thus for $z,z' \in I$, we see that
$\mu(z,z')=0$ implies that $e_1 \le |z-z'| \le A_0$, and, in general, since the gaps $e_k$ are decreasing,
%
%
$\mu(z,z')=\ell$ implies that $e_{\ell+1} \le |z-z'| \le A_\ell$.

Given $z \in E_k$ and $\ell \in \{0, \dots, k\}$ there are exactly $2^{k-\ell}$ points  $z' \in E_k$ with $\mu(z,z')=\ell$ and so for fixed $z \in E_k$ we have $$\prod_{\overset{z' \in E_k}{z' \neq z}} |z-z'| =\prod_{\ell=0}^k \prod_{\overset{z' \in E_k}{\mu(z,z')=\ell}} |z-z'| \ge \prod_{\ell=0}^k (e_{\ell+1})^{2^{k-\ell}}.$$

Thus $$P_{2^{k+1}}^2\ge \prod_{z \in E_k}\prod_{\overset{z' \in E_k}{z' \neq z}} |z-z'| \ge \prod_{z \in E_k} \prod_{\ell=0}^k (e_{\ell+1})^{2^{k-\ell}}=\left(\prod_{\ell=0}^k (e_{\ell+1})^{2^{k-\ell}}\right)^{2^{k+1}},$$
and so
\begin{equation}\label{logP}
\log P_{2^{k+1}}^2 \ge 2^{k+1}\sum_{\ell=0}^k 2^{k-\ell} \log(e_{\ell+1}).
\end{equation}

We show Cap $I>0$ by
using~\eqref{gap} to assert $e_{k+1} \geq B A_k$ and then computing
$$\log(\textrm{Cap }I)=\lim_{k \to \infty} \log D_{2^{k+1}} =\lim_{k \to \infty} \frac{1}{2^{k+1}(2^{k+1}-1)} \log P_{2^{k+1}}^2
=\lim_{k \to \infty}  2^{-2k-2} \log P_{2^{k+1}}^2 $$
$$\ge \sum_{\ell=0}^\infty 2^{-\ell-1} \log(e_{\ell+1})
\ge \sum_{\ell=0}^\infty 2^{-\ell-1} \log (B A_\ell)
=\sum_{\ell=0}^\infty 2^{-\ell-1} \left( \log B + \frac{(\ell^2+\ell)\log a}{2} \right)>-\infty.$$
%
%
%
\end{proof}

\begin{remark}
In Example~\ref{hdim=0} one can instead apply Frostman's Theorem to show Cap $I>0$ and thus forgo the given computation.
In this case, we can take $h(t)=\exp(-c \sqrt{-\log t})$ for small enough $t>0$ and a suitable constant $c>0.$  We omit the details. 
\end{remark}

\begin{proof}[Proof of (c) in Example~\ref{hdim=0}]
This follows from Theorem~\ref{hnup}(1) since $a_k = a^k \to 0$.
\end{proof}



We now present the proof of Theorem~\ref{dim=2} after a few preliminaries.

For each $m=2, 3, \dots$ let $W_m$ be the set $I$ constructed in Example~\ref{hdim} using $a=1/(m+1)$ and set $G_m = W_m \times W_m$.  Note that by the Product Formulas 7.2 and 7.3 of~\cite{Falc}, we have $2 \log{m}/\log(m+1) = \dim_H W_m+\dim_H W_m \leq \dim_H G_m \leq \dim_H W_m+\overline{\dim}_B W_m \leq \log{m}/\log(m+1) +1 <2$ where $\overline{\dim}_B$ denotes the upper box counting dimension.


\begin{proposition}\label{Cross}
Let $m \in \{2, 3, \dots\}$ and $\bar{a}=(a_1, a_2, \dots) \in (0, \frac{1}{m+1}]^\mathbb{N}$, with $\liminf a_k = 0$ (so that $I_{\bar{a}}$ given as in~\eqref{Idef} is HNUP by Theorem~\ref{hnup}).   Then $E= I_{\bar{a}} \times I_{\bar{a}}$ is HNUP.
\end{proposition}

\begin{proof}
Select a subsequence $a_{k_n}$ such that $a_{k_n} \to 0$.  Consider a point $(x,y) \in E$.  Now consider a basic interval $J_x$ of $I_k$ (in the construction of $I_{\bar{a}}$) containing $x$ and a basic interval $J_y$ of $I_k$ containing $y$.  As before we let $x_{J_x}$ denotes the midpoint of $J_x$, $x_{J_y}$ denotes the midpoint of $J_y$, $r_k=A_k/2$, and $R_k=r_k +e_k$.  Hence we see that $S=([x_{J_x}-R_k, x_{J_x}+R_k] \times [y_{J_y}-R_k, y_{J_y}+R_k]) \setminus (J_x \times J_y)$ is a (square shaped) annulus that separates $E$.  The picture quickly shows that $S$ contains the true annulus $C=Ann((x_{J_x},y_{J_y});\sqrt{2}r_k,R_k)$ which also separates $E$.

Since as shown in the proof of Theorem~\ref{hnup}(1), $R_{k_n}/r_{k_n} \to +\infty$ (with $R_{k_n} \to 0$), we see that for each point $(x,y) \in E$ there exists an annulus $C$ of arbitrarily large modulus with arbitrarily small outer radius which separates $E$ and has $(x,y)$ in the bounded component of the complement of $C$. Hence, $(x,y)$ can never be a point in a uniformly perfect subset of $E$ and thus no subset of $E$ can be uniformly perfect.
\end{proof}

\begin{proof}[Proof of Theorem~\ref{dim=2}]
Choose $0<\rho_n<1$ such that $\rho_n$ strictly decreases to 0 as $n \to \infty$.
Consider the annuli $B_n = Ann(0;\rho_{n+1},\rho_n)$.  For each $m = 2, 3, \dots$ let $G_m'$ be a translated and scaled down copy of $G_m$ such that $G_m' \subset B_{2m}$.  Note that each $G_m'$ is HNUP because each $G_m$ is HNUP and uniformly perfect sets remain uniformly perfect after translation and scaling.

Call $E= \{0\} \cup \bigcup_{m=2}^\infty G_m'$ noting that it is compact and also HNUP since the annuli $B_{2m+1}$ separate the HNUP sets $G_m'$ from each other.  Since $\dim_H G_m'=\dim_H G_m \geq 2 \log{m}/\log(m+1)$, we must have $\dim_H E = \sup_m \dim_H G_m' = \sup_m \dim_H G_m =2$.

Note that $m_2(E)=0$ since for all $m$ we have $m_2(G_m')=0$ because $\dim_H G_m'<2$.

Note that, due to the construction of the $G_m$, each point in $E \setminus \{0\}$ satisfies the the additional property of the theorem, and by choosing $\rho_n$ so that $\rho_n/\rho_{n+1} \to \infty$ as $n \to \infty$ we can ensure that that $0$ does too.
\end{proof}

\begin{proof}[Proof of Theorem~\ref{dim=1}]
Again, choose $0<\rho_n<1$ such that $\rho_n$ strictly decreases to 0 as $n \to \infty$.
Consider the annuli $B_n = Ann(0;\rho_{n+1},\rho_n)$.  For each $m = 2, 3, \dots$ let $W_m'$ be a translated and scaled down copy of $W_m$ such that $W_m' \subset B_{2m} \cap [0, 1]$.  Note that each $W_m'$ is HNUP because each $W_m$ is HNUP.

Call $W= \{0\} \cup \bigcup_{m=2}^\infty W_m'$ noting that it is compact and also HNUP.  Since $\dim_H W_m'=\dim_H W_m =  \log{m}/\log(m+1)$, we must have $\dim_H W = \sup_m \dim_H W_m' = \sup_m \dim_H W_m =1$.

Note that $H^1(W)=0$ since for all $m$ we have $H^1(W_m')=0$ because $\dim_H W_m'<1$.

In the same manner as in the proof of Theorem~\ref{dim=2}, we can also ensure that $W$ satisfies the the additional property of the theorem.
\end{proof}

\begin{remark}
  Note that $W \times W$ where $W$ is as constructed in the proof of Theorem~\ref{dim=1} is a candidate for satisfying the statement of Theorem~\ref{dim=2}, but it is not clear if it is HNUP.  The proof of Proposition~\ref{Cross} worked because of the high degree of uniformity in the sets $W_m$, that is, given $(x, y) \in W_m \times W_m$, for each basic set $J_x$ containing $x$ there was a basic set $J_y$ containing $y$ of the (exact) same size and both $J_x$ and $J_y$ had the exact same gap $e_n$ around them.  We do not have such a uniform structure in $W \times W$.  It would be interesting, however, to be able to settle the following related question.
\end{remark}

\begin{question}\label{ExE'}  Must $E \times E'$ be a HNUP subset of $\C$ when each $E$ and $E'$ are HNUP subsets of $\R$?
\end{question}

%

\vskip.1in

In addition to our thanks already expressed regarding Examples~\ref{m1>0} and~\ref{m2>0}, we thank here both the first and second referees for helpful comments that improved the presentation of this paper.

\vskip.1in

Acknowledgement.  This work was partially supported by a grant from the Simons Foundation (\#318239 to Rich Stankewitz). The research of the third author was partially supported by
JSPS KAKENHI 24540211, 15K04899.

\bibliographystyle{plain}

\bibliography{kyoto}

\end{document}